\documentclass[oneside,english]{amsart}
\usepackage[T1]{fontenc}
\usepackage[latin9]{inputenc}
\usepackage{url}
\usepackage{amsthm}
\usepackage{amssymb}

\makeatletter
\numberwithin{equation}{section}
\numberwithin{figure}{section}
  \theoremstyle{remark}
  \newtheorem*{rem*}{Remark}
 \theoremstyle{definition}
 \newtheorem*{defn*}{Definition}
  \theoremstyle{plain}
  \newtheorem*{thm*}{Theorem}
  \theoremstyle{remark}
  \newtheorem*{acknowledgement*}{Acknowledgement}
\theoremstyle{plain}
\newtheorem{thm}{Theorem}
  \theoremstyle{definition}
  \newtheorem{defn}[thm]{Definition}
  \theoremstyle{plain}
  \newtheorem{prop}[thm]{Proposition}
  \theoremstyle{remark}
  
 \theoremstyle{Lemma}
\newtheorem{Lemma}[thm]{Lemma}

\makeatother

\usepackage{babel}

\begin{document}

\title{Some geometric equivariant cohomology theories}

\author{Haggai Tene}
\begin{abstract}
In this paper we give a geometric construction of the Borel equivariant (co)homology for spaces with a $G$-action, where $G$ is a compact Lie group with the property that the adjoint representation is orientable. A nice feature of these constructions is that there are corresponding Poincar\'e dual (co)homology theories called backwards (co)homology. This gives rise to a third relative (co)homology theory which we call stratifold Tate (co)homology. These Tate groups agree with the original
definition of Tate cohomology for finite groups given by Swan.
All constructions in this paper are geometric and use stratifolds.
One advantage of this description is that elements in these groups
can be described concretely by representatives. We give some 
examples
of that.
\end{abstract}
\maketitle

\section{Introduction}

\begin{rem*}
{\em In this paper we denote by $G$ a compact Lie group with the property that the adjoint representation is orientable. All groups will be assumed to have this property unless stated otherwise.}
\end{rem*}

In this paper we give a geometric construction of an equivariant (co)homology theory called {\bf equivariant stratifold (co)homology} which under certain conditions is naturally isomorphic  to singular (co)homology of the Borel construction with integral coefficients. Due to the geometric nature of the construction as bordism groups, the cohomology groups are only defined for smooth oriented manifolds (or smooth, separable, infinite dimensional Hilbert manifolds) with a smooth $G$-action which is assumed to preserve the orientation. We call such manifolds oriented $G$-manifolds (or $G$-Hilbert manifolds).

There are several reasons for being interested in such a geometric description.  The first reason is that it allows new constructions, or sheds light on old constructions. Here are some examples of that: In this paper we use this geometric description to construct another equivariant (co)homology theory which we call {\bf backwards (co)homology}, which has the nice property that it is Poincar\'e dual to  equivariant stratifold (co)homology. Theories with this property (and the later discussed Tate groups) where constructed earlier by Greenlees \cite{Gr} for finite groups and Greenlees and May  \cite{GM} for compact Lie groups and are sometimes called  polar counterparts.  We stress here the geometric flavor. In another paper \cite{T2} we have used the geometric nature to give a description of the product in negative Tate cohomology of finite groups (and a generalization for compact Lie groups). Using this geometric description some interesting vanishing results were obtained.

An example of a new construction based on our geometric approach, which we discuss in this paper, is a simple description of a product in the equivariant cohomology of $G$ with the conjugation action, denoted by $G^{ad}$. This gives a product in the cohomology of $LBG$, the free loop space in $BG$, due to the homotopy equivalence $G^{ad}\times _G EG\to LBG$.
Finally we use the geometry to give an example of an explicit constructions of (co)homology classes and the computation of induced maps and cup products.  

After defining backwards (co)homology, we give a geometric constructions of a natural transformation from backwards cohomology  to  equivariant stratifold cohomology, and a natural transformation from  equivariant stratifold homology to backwards homology. Using these natural transformations, a relative (co)homology theory is defined, which is called {\bf stratifold Tate (co)homology}. All these groups fit into long exact sequences. 

To give a flavor of the constructions we explain the stratifold homology groups of a $G$-space, a space with continuous $G$-action, $X$. The basic geometric input are stratifolds, which are certain stratified spaces. A key property of stratifolds is that they are differential topological objects where most of the basic results from smooth manifolds, like Sard's theorem or transversality, hold. This allows the construction of bordism groups in the same way as with smooth manifolds. 

An equivariant homology class of degree $k$ is given by a compact, regular, oriented $k$ dimensional stratifold $S$ with a free, orientation preserving $G$-action together with an equivariant map to $X$. Details will be explained later. Two such maps represent the same homology class if and only if they are bordant. The resulting bordism groups are the {\bf equivariant stratifold homology groups}
$$
SH_k^G(X).
$$

Similar groups were constructed by MacPherson \cite{Mac} using pseudomanifolds instead of stratifolds.

In case $G$ is trivial, Kreck \cite{Kr} has shown that these bordism groups, with induced maps given by composition, form a homology theory which fulfills the Eilenberg Steenrod axioms and so, for $CW$ complexes, agree with singular homology. Following an idea of Quillen, Kreck \cite{Kr} also  defined  geometric cohomology groups for smooth oriented manifolds and showed that they agree with singular cohomology with integral coefficients. Recently, Kreck and the author \cite{K-T} defined a geometric cohomology theory for Hilbert manifolds using a generalization of stratifolds to the setting of Hilbert manifolds. This plays an important role in this paper since we will use a Hilbert manifold model for the classifying spaces $BG$ and $EG$  to define equivariant versions of this theory. The resulting cohomology theory, called {\bf equivariant stratifold cohomology}, is denoted by 
$$
SH^*_G,$$
and is defined for $G$-Hilbert manifolds. The cohomology classes of a $G$-Hilbert manifold $P$  are represented by equivariant, proper, Fredholm maps from $G$-Hilbert stratifolds $S$ to $P \times EG$, where the degree of the cohomology class is minus the Fredholm index of the restriction of the map to the top stratum of $S$. Details are explained later. The relation to the (co)homology of the Borel construction is given by the following theorem:

\begin{thm}
There is a natural isomorphisms of equivariant homology theories 
$$SH_{k}^{G}(X)\to H_{k-dim(G)}^{G}(X;\mathbb{Z})
$$ 
for $G$-CW complexes $X$ and 
 a natural isomorphism of multiplicative
equivariant cohomology theories  
$$
SH_{G}^{*}(P) \to H_{G}^{*}(P;\mathbb{Z})
$$ 
for $G$-Hilbert manifolds $P$.
\end{thm}

Poincar\'e duality between the cohomology and homology groups of the Borel construction of a closed oriented $G$-manifold fails, for example if the manifold is a point. From the point of view of stratifold homology and cohomology for closed oriented smooth manifolds there are almost tautological Poincar\'e dual groups, where both homology and cohomology  are given as bordism theories. The only difference is that homology classes are represented by compact stratifolds and cohomology classes by proper maps from (not necessarily compact) stratifolds. Applying this principle to the equivariant stratifold (co)homology theories constructed above we obtain two new theories, the above mentioned {\bf backwards (co)homology}
$$
DSH^*_G(M),
$$
defined for oriented $G$-manifolds, and 
$$
DSH_*^G(X),
$$
defined for $G$-spaces. For closed oriented $G$-manifolds Poincar\'e duality holds:

\begin{thm}
Let $M$ be a closed oriented $G$-manifold of dimension $m$, then there are isomorphisms:
$$
PD:DSH_{G}^{k}(M)\rightarrow SH_{m-k}^{G}(M),
$$ and
$$
PD:SH_{G}^{k}(M)\rightarrow DSH_{m-k}^{G}(M).
$$ 

\end{thm}

For an oriented $G$-manifold there is a natural transformation
$$
DSH^k_G(M) \to SH^k_G(M)
$$
assigning to an equivariant, proper map $f: S \to M$ the map $ f \times \,\, id : S \times EG \to M \times EG$. In such a situation one can construct a third cohomology theory where the cohomology classes are represented by stratifolds $T$ with boundary of the form $S \times EG$ where $S$ is an oriented, free $G$-stratifold, mapping to $M \times EG$ such that the restriction to the boundary is the product of a proper map $S \to M$ with the identity map on $EG$. This cohomology theory is called {\bf stratifold  Tate cohomology}
$$
\widehat {SH}^*_G(M)
$$
and is defined for oriented $G$-manifolds. We similarly construct {\bf stratifold Tate homology} for $G$-spaces
$$
\widehat {SH}_*^G(X).
$$ 
These Tate groups fit into a long exact sequence:

\begin{thm}
For an oriented $G$-manifold $M$ and a $G$-space $X$ there  are long exact sequences
$$
...\to DSH_{G}^{k}(M)\to SH_{G}^{k}(M)\to\widehat{SH}_{G}^{k}(M)\to DSH_{G}^{k+1}(M)\to... ,
$$
and
$$
...\to SH^{G}_{k}(X)\to DSH^{G}_{k}(X)\to\widehat{SH}^{G}_{k}(X)\to SH^{G}_{k-1}(X)\to... .
$$

\end{thm}

This implies that for a closed oriented $G$-manifold $M$, the groups $\widehat{SH}_{G}^{k}(M)$
are the obstructions to Poincar\'e duality in the various dimensions.

There is another reason why the Tate groups are interesting. They are also an obstruction for orientation preserving free actions on a smooth manifold:

\begin{thm}
Let $M$ be an oriented $G$-manifold then $\widehat{SH}_{G}^{k}(M)$ vanishes for all $k\in \mathbb Z$
if and only if $G$ acts freely on $M$.
\end{thm}

\begin{rem*} {\em Greenlees \cite{Gr} for finite groups and Greenlees and May  \cite{GM} for compact Lie groups defined equivariant (co)homology theories which are related to the (co)homology theories of the Borel construction in a similar way as backwards and stratifold Tate cohomology. They also fit into an exact sequence as in Theorem 3. They use equivariant spectra to carry this out. It is expected that our theories are isomorphic to their theories for groups with the property that the adjoint representation is orientable but it is not so easy to build a bridge from theories constructed using spectra to theories constructed using stratifolds. Even for the trivial group Kreck has shown that stratifold (co)homology agrees with singular (co)homology rather indirectly through a characterization of the theories by axioms \cite{Kr}. It is desirable and a very interesting question to identify the equivariant (co)homology theories defined in this paper with those defined by Greenlees and May. At least for finite groups the Tate cohomology groups agree, since both groups are isomorphic to the original Tate cohomology groups defined by Swan. This was proved by Greenlees in \cite{Gr} for his groups and for the group presented here in \cite{Te}:
\begin{thm}
Let $G$ be a fixed finite group. On the category of oriented $G$-manifolds
there is an isomorphism $\widehat{SH}_{G}^{k}(M)\to\widehat{H}_{G}^{k}(M)$
of groups for all $M$ and $k$. 
\end{thm} 
We will not repeat the rather technical proof of this Theorem and refer to \cite{Te} instead. Once we are able to identify the Tate groups defined in this paper with those defined by Greenlees and Greenlees and May this will be a consequence of Greenlees' identification of his theory with Swan's theory.} 
\end{rem*}

One can also define analogous theories which correspond to equivariant cohomology groups with $\mathbb Z/2$-coefficients. For this we need to use stratifolds such that the top stratum is not oriented. In that case, one obtains {\bf equivariant (co)homology theories with $\mathbb {Z}/2$ coefficients}
$$SH^*_G(M;\mathbb{Z}/2), DSH^*_G(M;\mathbb{Z}/2), \widehat{SH}^*_G(M;\mathbb{Z}/2), $$

\noindent defined for $G$-manifolds $M$, and

$$SH_*^G(X;\mathbb{Z}/2), DSH_*^G(X;\mathbb{Z}/2), \widehat{SH}_*^G(X;\mathbb{Z}/2), $$

\noindent defined for $G$-spaces $X$. 

\begin{thm}
Theorems 1-3 hold for the equivariant (co)homology theories with $\mathbb {Z}/2$ coefficients. 
\end{thm}

One of the motivations for geometric definitions of (co)homology theories is to make things more concrete. The aim of this paper is mainly to construct the theories and to prove the fundamental theorems but we briefly address this question in the last section where we give some concrete computations.

The paper is structured as follows. In Section 2 we give the definition of  equivariant stratifold (co)homology  and prove Theorem 1. In Section 3 we define backwards (co)homology and prove Theorem 2. In Section 4 we define stratifold Tate (co)homology and prove Theorem 3 and 4. In Section 5 we give an example where the groups
$\widehat{SH}_{G}^{k}(M)$ are computed and concrete generators are found.
We use Poincar\'e duality to describe a product in the cohomology of the
free loop space of the classifying space $BG$. 

This paper is based  in part on the author's PhD thesis \cite{Te}. Some of the details which are being discussed here can be found in more detail in the original text. This thesis was written under the
direction of Matthias Kreck at the Hausdorff Research Institute for
Mathematics (HIM) in the University of Bonn. The author would like
to thank Matthias Kreck for his support, including sharing some of his ideas regarding the construction of some of the theories, the Hausdorff institute and the Hausdorff Center for Mathematics (HCM) for
financial support. A part of the writing of this paper was done in
Postech.

\section{Stratifolds and  equivariant stratifold homology and cohomology}

We begin with the definition of stratifolds given by Kreck \cite{Kr}. He defines them as topological spaces equipped with a sheaf of real continuous functions fulfilling certain conditions. In this language, a smooth manifold is defined as a space $M$ together with a sheaf which is locally isomorphic to $\mathbb R^m$ equipped with the sheaf of smooth real functions. Thus, for a smooth manifold, the sheaf is given by the sheaf of smooth real functions on $M$. 

We prepare for the definition of stratifolds with the introduction of some notation. 

Let $(S,F)$ be a pair consisting of a topological space $S$ together with a subsheaf of the sheaf of real continuous functions $F$. Denote: 
\begin{itemize}
\item $T_{x}S$, the tangent space at a point $x\in S$ - the vector space
of derivations of germs $\Gamma_{x}(F)$. 
\item $S_{k}=\left\{ x\in S|dim(T_{x}S)=k\right\} $ - the $k^{th}$ stratum.
\item $\Sigma_{k}=\cup_{i\leq k}S_{i}$ - the $k^{th}$ skeleton.\end{itemize}
\begin{defn}
{\em An {\bf n dimensional stratifold} is a pair $(S,F)$ where $S$ is a
locally compact, Hausdorff space with countable basis and $F$ is
a subsheaf of the sheaf of continuous real functions, called smooth
functions, fulfilling the following properties:
\begin{itemize}
\item The restriction of $F$ to $S_{k}$ gives $S_{k}$  the structure of a smooth
manifold.
\item The restriction map $\Gamma_{x}(F)\to\Gamma_{x}(F|_{S_{k}})$ of germs
to each stratum is an isomorphism. 
\item $S=\Sigma_{n}$.
\item All skeleta are closed. 
\item There exists subordinated  smooth partitions of unity.
\end{itemize}
If $S$ and $S'$ are stratifolds, a {\bf morphism (or a smooth map)} between
them is a continuous map $f:S\to S'$ such that the pullback of smooth
maps is smooth. We often call isomorphisms {\bf diffeomorphisms}.

An {\bf n dimensional stratifold with boundary} is a pair of topological spaces $(T,\partial T)$, where $\partial T$ is a closed subspace, together with the structure of an $n$ dimensional stratifold on $T\setminus\partial T$
and the structure of an $(n-1)$ dimensional stratifold on $\partial T$, together with
a germ of collars, where a collar is a homeomorphism $c:\partial T\times[0,\varepsilon)\to U$,
where $U$ is an open neighbourhood of $\partial T$ in $T$, such that $c$
restricted to $\partial T\times(0,\varepsilon)$ gives a diffeomorphism
onto its image.}
\end{defn}

For some fundamental properties and constructions of stratifolds see Kreck \cite{Kr}.

\noindent {\bf Example:} One of the most important features of stratifolds is that the cone over a stratifold is a zero bordism. Here the {\bf cone} over $(S,f)$ is defined as $(S \times [0,1]/_{S \times \{0\}},C(F))$, where the smooth functions in $C(F)$ are defined as those whose restriction to $S \times (0,1)$ are smooth functions in the product stratifold structure and which are locally constant near the cone point. The last property is equivalent to the condition $\Gamma_{x}(F)\cong \Gamma_{x}(F|_{S_{0}})$ for the cone point $x$. 

This implies that every stratifold is zero bordant. To avoid that all (co)homology theories are trivial we will require that the codimension $1$ stratum is always empty. Then a zero dimensional stratifold, which is a zero dimensional manifold, is in general  not zero bordant. To control orientation we also equip the top stratum with an orientation:
\begin{defn}
{\em A stratifold is {\bf oriented} if the codimension $1$ stratum is empty and the top stratum is oriented. A stratifold with boundary is oriented if the interior and the boundary are oriented and the collar is orientation preserving, where we equip the cylinder over the boundary with the product orientation. }
\end{defn}

After we have defined oriented stratifolds and stratifolds with boundary one can define bordism groups as for oriented manifolds. We will generalize this to the equivariant setting. 

\begin{defn} {\em A {\bf $G$-stratifold} is given by a smooth action of $G$ on $S$, where smooth means that $G \times S \to S$ is smooth. An {\bf oriented, free $G$-stratifold} is a $G$-stratifold such that the action is free and orientation preserving. A $G$-stratifold is called {\bf regular} if for each $x$ in the $k^{th}$ stratum $S_k$ there is a $G$-invariant open neighbourhood $V$ in $S$ and $G$-stratifold $F$ with trivial action such that $V$ is equivariantly diffeomorphic  to $(V \cap S_k)   \times F$.}
\end{defn}

We will only consider regular, oriented, free $G$-stratifolds. Together with equivariant maps they will be representatives of our (co)homology classes. 

\begin{rem*}
{\em  One might define oriented, free $G$-stratifolds in a stronger sense, so that $\pi_0(G)$ acts orientation preserving on $S/G_1$, where $G_1$ is the component of the identity. This is equivalent to the condition that the action is orientation preserving since $G$ has the property that the adjoint representation is orientable, but for other compact Lie groups these two conditions are different (for example, take $G=O(2)$ acting by left multiplication on $S=SO(3)$). This suggests that there are at least two ways to generalize the theories which appear in this paper to compact Lie groups. In both cases the stratifold Tate groups will not agree with the Tate groups in \cite{GM}. This point was overlooked in \cite{Te}, and some of the statements there are not correct, unless one restrict the groups like here, or need reformulation. 
}
\end{rem*}

We begin with the definition of  equivariant stratifold homology. 

\begin{defn}{\em 
Let $X$ be a $G$-space. We denote the bordism classes of equivariant maps $f: S \to X$, where $S$ is a compact, regular, oriented, free $G$-stratifold by 
$$
SH_{k}^G(X),
$$
the {\bf $k^{th}$ equivariant stratifold
homology group} of $X$.  Addition is given  by disjoint union. Induced maps are given by composition. } 
\end{defn}

Exterior product is given by Cartesian product. The construction of the boundary operator and the proof of Mayer-Vietoris for equivariant open subsets is similar to \cite{Kr}, giving it the structure of an equivariant homology theory.

Quillen \cite{Qu} has given a geometric description of complex cobordism groups but only for smooth manifolds. Kreck has used this idea to give a geometric interpretation of singular cohomology for smooth oriented manifolds using stratifolds. We would like to generalize this to give a cohomology theory $SH^k_G(M)$ for oriented $G$-manifolds $M$, which is naturally isomorphic to the cohomology of the Borel construction $H^k_G(M,\mathbb Z)$. This is not completely obvious. Let us indicate the difficulties. If there was a finite dimensional manifold homotopy equivalent to $BG$, we would proceed as follows. We would consider bordism classes of (non compact) regular, oriented, free $G$-stratifolds $S$ together with an equivariant, proper, smooth map  $f: S \to M \times EG$. But there are no finite dimensional manifold models for $BG$, only models given by Hilbert manifolds. Thus it is desirable to extend the whole setting to Hilbert manifolds. This was carried out in joint work with Kreck \cite{K-T}. We summarize the concepts. As before a Hilbert stratifold is a topological space $S$ together with a sheaf $F$ of continuous functions. One can also consider the tangent spaces but one cannot use their dimension to define the strata. Since in the end we are interested in smooth maps from a Hilbert stratifold to a Hilbert manifold we only define singular Hilbert stratifolds in a Hilbert manifold $P$. We indicate here the basic idea and refer for details to \cite{K-T}. We require that the differential of $f: S \to P$ has finite dimensional kernel and cokernel at each point of $S$. Thus we can speak about the Fredholm index at each point and the strata are defined as those points where the index is fixed. In the finite dimensional setting these are precisely the strata as defined above. Once one has the strata all other concepts of stratifolds can be verbally generalized to the infinite dimensional setting to give the definition of {\bf singular Hilbert stratifolds}.  We say that a singular Hilbert stratifold has degree $k$, if there are no points $x \in S$ where the Fredholm index is larger than $-k$. There is also the concept of an orientation of a singular Hilbert stratifold in terms of the determinant line bundle roughly given by the tensor products of the kernel and cokernel bundle on the top stratum. This is a bit delicate and details can be found in \cite{K-T}.

Now we proceed as above for homology and define singular $G$-stratifolds in a Hilbert manifold $P$. With them we define the equivariant cohomology groups $SH^k_G(P)$ as follows:

\begin{defn}{\em 
Let $P$ be a $G$-Hilbert manifold. We denote the bordism classes of equivariant, regular, oriented, free, singular $G$-Hilbert stratifolds $f: S \to P \times EG$, where $f$ is of degree $k$, by
$$
SH^{k}_G(P),
$$
the {\bf $k^{th}$ equivariant stratifold
cohomology group} of $P$.  The addition is given by disjoint union. }
\end{defn}
 
Induced maps are not so easy in this setting. But  if $f: P \to L$ is an equivariant submersion, then one can define induced maps by pull back. In particular they are defined for projections of smooth $G$-vector bundles. In this situation the induced map is an isomorphism. This follows from the fact that it holds for ordinary equivariant cohomology and we will give in the proof of Theorem 1 a direct identification with the stratifold cohomology groups which commutes with the induced maps for submersions. Then one can define induced maps for general maps by writing the map as the composition of the embedding $(Id, f):P \to P \times L$  and the projection  $\pi_L: P \times L\to L$. Then one considers an equivariant tubular neighborhood $U$ of $P\subseteq P \times L$, pulls the class back to $P \times L$, restricts it to $U$ and composes it with the inverse of the map induced by the projection $U \to P$. This way we obtain a contravariant functor. A coboundary operator is defined in an analogous way to the non equivariant way, and exterior product by Cartesian product. The Mayer Vietoris sequence for invariant open subsets holds as before. $SH^*_G$ is a multiplicative equivariant cohomology theory on the category of $G$-Hilbert manifolds.

Now we prepare the proof of Theorem 1 with a Lemma:

\begin{Lemma}
a) Let $S$ be a free $G$-stratifold. The sheaf on $S/G$, given by the functions on open subsets of $S/G$ which pull back to smooth functions on their preimage in $S$ defines a  stratifold structure on the orbit space.  Similarly if $f:S \to P$ is an equivariant singular Hilbert $G$-stratifold and $G$ acts freely on $S$ and $P$, then there is an induced singular Hilbert stratifold $f/G:S /G \to P/G$. The projection $p: S \to S/G$ is a smooth principal bundle, i.e. the local trivializations are isomorphisms. If $S$ is $G$-regular then $S/G$ is regular.

\noindent b) If $p: \tilde S \to S$ is a smooth $G$-principal bundle over a stratifold $S$ resp.\ $p: \tilde P \to P$ is a smooth principal bundle over the Hilbert manifold $P$ and $f: S \to P$ is a singular Hilbert stratifold then there is a stratifold structure on $\tilde S$ resp.\ a singular stratifold structure on $\tilde f: f^*(\tilde P) \to \tilde P$ such that if we pass to the orbit space we obtain $S$ resp.\ the singular Hilbert stratifold $f: S \to P$ back. And if we start with the situation in a), pass to the orbit space and go again to the total space of the principal bundles we obtain the original situation again. If $S$ is regular then $\tilde{S}$ is $G$-regular.

\noindent c) If $S$ is an oriented, free $G$-stratifold resp.\ $G$ acts on $f:S \to P$ orientation preserving, then there is an induced orientation on the orbit space, and an orientation on the orbit space gives rise to an orientation on the total space such that the action is orientation preserving.
\end{Lemma}

\begin{proof} a) $(S/G)_k=S_{k+dim(G)}/G$. This follows from the equivalent statement for manifolds. For extension of germs, we can use the extension of the lift to $S$ and then  average, using the fact that $G$ is compact. Uniqueness of extension in $S/G$ follows from uniqueness of extension in $S$. $S/G$ has strata up to dimension $dim(S)-dim(G)$. $G$ acts by diffeomorphisms, so it acts on each stratum separately. Therefore, the skeleta in $S$ are invariant (and closed), thus their image in $S/G$ is closed. For the partition of unity, in $S/G$ we find the partition of unity in $S$ and average. The case of Hilbert stratifolds is similar. Notice, that in the first case the dimension of the stratifold is decreased by the dimension of $G$, but in the second case, the Fredholm index is unchanged.  The assertion about regularity is clear.

\noindent b) One way to get the structure of a stratifold on $\tilde{S}$ is to look at a classifying map $f:S\to BG$ for the bundle $\tilde{S}\to S$. This can be approximated by a smooth map. This gives $\tilde{S}$ a structure of a stratifold by pullback. This is independent of the choice of $f$ since every two such maps are smoothly homotopic. The same argument works for Hilbert stratifolds as well. 

If $S$ is a free $G$-stratifold, then there is a classifying map $f:S\to EG$, such that the following diagram is a smooth pull back diagram:

\[
\begin{array}{ccc}
S & \to & EG\\
\downarrow &  & \downarrow\\
S/G & \to & BG \end{array}\]

\noindent This implies that if we start with $S$ then define $S/G$ and give $S$ the smooth structure described here we get the original structure we started with. If we start with $S/G$, define the structure on $S$ and go back we get again $S/G$ since the smooth maps on $S/G$ are exactly those maps which lift to smooth maps on $S$. The assertion about regularity is clear.

\noindent c) An orientation of a stratifold is given by an orientation of its top stratum. Therefore, to orient $S/G$ it is enough to assume $S$ is a manifold, which we denote by $M$. We want to orient the normal bundle of the orbits in a continuous way ($M\to M/G$ is a locally trivial bundle). Since $M$ is oriented, this is equivalent to orienting all orbits in a continuous way. Fix an orientation of $T_1 G$, the tangent space of $G$ at the identity, and orient $G$ by left translations, which we denote by $l_g (k)=gk$. Denote by $r_g (k)=kg$  and by $c_g (k)=g^{-1}kg$.

Let $O \subseteq M$ be an orbit. For each $x\in O$ the diffeomorphism $f_x:G\to O$  given by $f_x(g)=g\cdot x$ induces an orientation of $O$. The orientation we get in a point $z=h\cdot x$ is given by the differential of the composition $f_x \circ l_h$ applied to the orientation of $T_1 G$. We want to check the dependency in $x$. If we choose $y\in O$ then there exists a unique $g\in G$ such that $y=g\cdot x$ then the orientation in $z$ is given by the differential of the composition $f_y \circ l_{hg^{-1}}$ applied to the orientation of $T_1 G$. 
But $f_y \circ l_{hg^{-1}}=  f_x \circ r_g \circ l_{hg^{-1}}= f_x \circ l_{h} \circ c_{g}$. Since the adjoint representation is orientable both orientations agree, and we are done.
\end{proof}

Now we are ready to prove Theorem 1:

\begin{proof} (Theorem 1) Fix a model for $EG$ which is a $G$-CW complex. Let $X$ be a $G$-CW complex. Given a  class  $[f: S \to X]\in SH_k^G(X)$, since $G$ acts freely on $S$, which is paracompact, there is an equivariant map $g:S \to EG$. This defines a map $(f,g): S \to X \times EG$, where the right side is equipped with the diagonal action of $G$. We divide by the $G$ action to obtain a homology class in $SH_{k - dim (G)}(X \times _G EG)$. Since $g$ is determined up to homotopy, its choice does not change the bordism class. This defines a homomorphism
$$
SH_k^G(X) \to SH_{k - dim (G)}(X \times _G EG)
$$
since the same can be done for the bordism. By Lemma 12, this map is an isomorphism. This construction is natural and commutes with the boundary operator.

Now, by theorem $20.1$ in \cite {Kr}, for all $CW$ complexes $Y$ there is an isomorphism  $SH_k(Y) \cong H_k(Y, \mathbb Z)$, in particular $SH_k(X \times _G EG) \cong H_k(X \times _G EG, \mathbb Z)$ in a natural way.  Combining the two isomorphisms we obtain a natural isomorphism of equivariant homology theories
$$
SH_k^G(X) \to H^G_{k - dim (G)}(X, \mathbb Z).
$$

We now pass to cohomology. Fix a $G$-Hilbert manifold model for $EG$ (for existence see  \cite{Ee}). Applying the same idea one obtains an isomorphism $SH^k_G(P) \cong H^k(P \times _G EG, \mathbb Z)$, where $P$ is a $G$-Hilbert manifold. Here it is a bit simpler since the map into $EG$ is already given. Note that here there is no shift of degrees since the Fredholm index on the total space and orbit space are the same. Then we apply \cite{K-T} to identify $SH^k(P \times _G EG)$ with $H^k(P \times _G EG, \mathbb Z)$ and so obtain an isomorphism
$$
SH^k_G(P) \to H^k_G(P).
$$
This isomorphism commutes with maps induced by equivariant submersions. Since the construction of induced maps for general equivariant maps is reduced to the case of submersions this proves naturality. The construction of the coboundary is analogous to the one in the non-equivariant case. Therefore, this natural transformation commutes with it. It also commutes with the cup product, hence we get a natural isomorphism of multiplicative equivariant cohomology theories.
\end{proof}

\section{The dual theories - equivariant backwards homology and cohomology}

~

We describe now equivariant homology and cohomology theories following a principal introduced by Kreck \cite{Kr}, generalizing an idea of Quillen \cite{Qu}. 
We first define backwards cohomology groups. In our geometric context those are easily defined. Assume we are given a homology theory given by bordism classes of maps from compact stratifolds to a space $X$. The corresponding cohomology theory is defined for oriented manifolds, and is given by bordism classes of proper, smooth maps, where we don't require that the stratifold is compact. If $M$ is a closed manifold then properness of the map is equivalent to the fact that the stratifold is compact, and so we obtain essentially the same groups. This is a form of Poincar\'e duality, and the grading is made in such a way that Poincar\'e duality has the usual form.

\begin{defn}{\em 
Let $M$ be an oriented $G$-manifold of dimension $m$. Denote the set of bordism classes of
equivariant, proper smooth maps from regular, oriented, free $G$-stratifolds of dimension
$m-k$ by
$$
DSH_{G}^{k}(M).
$$}
\end{defn}

Addition is given by disjoint union, and the exterior product is given by Cartesian product. The definition of the coboundary and the proof of the Mayer Vietoris for invariant open subsets is similar to \cite{K-T}.  This cohomology theory is called {\bf equivariant backwards cohomology}.
Note that the groups can be non-trivial for negative $k$, and that in general $DSH_{G}^{*}(M)$ is a ring with no unit.

We explain how to define the induced maps in this case. We note two things about free actions. First, equivariant smooth approximation: an equivariant, continuous (proper) map $g:S\to N$  from a free $G$-stratifold to a free $G$-manifold is homotopic to a smooth equivariant (proper) map. To see this, we pass to the orbit map. Since the action is free the orbit spaces have the structure of a stratifold and a smooth manifold respectively.  Then use the non equivariant smooth approximation, as appears in \cite{Kr}. 
Second, equivariant transversal approximation: if in addition we have an equivariant smooth map $f:M\to N$ where $M$ is also a free $G$-manifold, then $g$ is homotopic to an equivariant proper smooth map which is transversal to $f$. To see this, we again pass to the orbit spaces and approximate $g/G$ by a transversal (proper) map. This can be done inductively over the skeleta, using regularity. If the map is transversal when restricted to the $k^{th}$ skeleton, then, by regularity, it is also transversal on an open neighbourhood of it. Then we approximate the map relative to some smaller neighbourhood of the $k^{th}$ skeleton. This could be done since the $(k+1)^{th}$ stratum is a smooth manifold. The lift of this map is transversal to $f$. 

We return to induce maps. Given an equivariant smooth map $f:M\to N$ between oriented $G$-manifolds, and an element represented by a proper map $g:S\to N$, where $S$ is an oriented, free $G$-stratifold as above. Let $EG'$ be a highly connected free $G$-manifold, such that there is an equivariant map $g':S\to EG'$, which can be assumed to be smooth. 
Look at the maps $f\times Id:M\times EG'\to N\times EG'$ and $g\times g':S\to N\times EG'$. Notice that $g$ is transversal to $f$ if and only if $f\times Id$ is transversal to $g\times g'$. Since now the action on all spaces is free, transversality can be achieved as noted above. Then, the composition of the approximated map $S\to N\times EG'$ with the projection on $N$ is transversal to $f$. Notice that those approximations can be done relative to a closed subset, hence we can apply them for bordisms rel.\ boundary.

For later use we prove the following

\begin{prop} Let $M$ be a free, oriented $G$-manifold. Then dividing by the $G$-action gives an isomorphism
$$
DSH_G^k(M) \cong H^k(M/G;\mathbb Z).
$$
\end{prop}

\begin{proof} By lemma 12 we can orient $M/G$, hence $SH^k(M/G)$ is defined. Define
$$
DSH_G^k(M) \to SH^k(M/G)
$$
by $[S\to M] \mapsto [S/G\to M/G]$. By lemma 12 the right side is well defined and the map is an isomorphism. The latter group is isomorphic to $H^k(M/G;\mathbb Z)$. 

\end{proof}

We now define backwards homology. 
\begin{defn}{\em 
Let $X$ be a $G$-space. For $k\in\mathbb{Z}$,
define $DSH_{k}^{G}(X)$ to be the set of bordism classes of equivariant
maps $(g_1,g_2):S\rightarrow X\times EG$ such that  $g_2:S\to EG$
 is a regular, oriented, free, singular $G$-Hilbert stratifold of index $-k$. Induced maps are defined by composition and exterior product by Cartesian product. }
\end{defn}

The proof of Theorem 2 is essentially straightforward:

\begin{proof} (Theorem 2) 
We first show that the map $PD:DSH_{G}^{k}(M)\rightarrow SH_{m-k}^{G}(M)$ is an isomorphism. An element in $DSH_{G}^{k}(M)$ is represented by an equivariant, proper, smooth map $f:S\to M$ where $S$ is a regular, oriented, free $G$-stratifold of dimension $m-k$. Since $M$ is compact and $f$ is proper, it follows that $S$ is compact, so this represents an element in $SH_{m-k}^{G}(M)$. Since every equivariant continuous map $f':S\to M$ is $G$-homotopic to a smooth one (as noted before), $PD$ is surjective, and since the same argument holds for bordisms relative to the boundary it is injective. 

We now show that $PD:SH_{G}^{k}(M)\rightarrow DSH_{m-k}^{G}(M)$ is an isomorphism. First we show that it is well defined. An element $SH_{G}^{k}(M)$ is represented by a regular, oriented, free, singular $G$-Hilbert stratifold $(f_1,f_2): S \to M \times EG$, where $(f_1,f_2)$ is of Fredholm index $-k$. $M$ is finite dimensional so $(f_1,f_2)$ is Fredholm of index $-k$ if and only if $f_2$ is Fredholm of index $m-k$, $M$ is compact so $(f_1,f_2)$ is proper if and only if $f_2$ is proper. Therefore, it represents an element in $DSH_{m-k}^{G}(M)$. If we show that every continuous equivariant map $f_1:S\to M$ can be approximated by an equivariant smooth map then we will be done, since it will follow that $PD$ is surjective, and by the same argument applied for the bordism relative to the boundary we show it is injective. 

To show that $f_1$ is equivariantly homotopic to a smooth equivariant map it is enough to show that the map $(f_1,f_2)$ is. But in this case the action on both sides is free so it is enough to show it on the orbit space. On the orbit space maps can be approximated by smooth maps using a partition of unity and the fact that every infinite dimensional Hilbert manifold can be smoothly embedded as an open subset of the Hilbert space. Notice that we do not require that this smooth approximation will be either proper or Fredholm. Regarding orientation, the fact that $M$ is oriented gives the map $(f_1,f_2):S\to M\times EG$ a natural orientation.

\end{proof}

\section{\label{sec:Stratifold-Tate-cohomology}Stratifold Tate cohomology}

As noted before, for an oriented $G$-manifold $M$ we define a multiplicative natural transformation 
 $\Theta:DSH^k_G(M)\to SH^k_G(M)$ by $[S\to M]\to[S\times EG\to M\times EG]$. 

Given such a natural transformation between two bordism theories one
can define a {}``relative term'' using maps from stratifolds with
boundary, which fits into a long exact sequence. We now define this
{}``relative term'' which we call stratifold Tate cohomology. It was shown in \cite{Te}, that it is isomorphic to Tate cohomology when $G$ is finite.

\begin{defn}{\em 
Let $M$ be an oriented $G$-manifold. Define $\widehat{SH}_{G}^{k}(M)$ to be the
set of bordism classes of maps of singular $G$-Hilbert stratifolds of
degree $k$ with boundary $[f:(T,S\times EG)\to M\times EG]$ $ $where
$S\times EG$ is the boundary of $T$ and the map $f$ restricted
to it is given as a product of maps $g:S\to M$ and $Id:EG\to EG$.
We require that $S$ and $T$ will come with a free, smooth and orientation
preserving $G$-action, and the maps will be equivariant. The bordism
relation is defined as follows: 

\[
[f:(T,S\times EG)\to M\times EG]\,\,\, and\,\,\,[f':(T',S'\times EG)\to M\times EG]\]
are bordant if and only if $[S\to M]$ is bordant to $[S'\to M]$
via some bordism $B$ and the map $[T\cup B\times EG\cup T'\to M\times EG]$
is bordant to an element of the form $[B'\times EG\to M\times EG]$. It is not hard to show that this is an equivalence relation.

With this definition $ $$\widehat{SH}_{G}^{k}$ can be given the
structure of an equivariant cohomology theory. Induced maps are given
by pullback and the coboundary operator is defined in a similar way
to the one for $SH^{k}(M)$.}
\end{defn}

A natural transformations $SH_G^{k}(M)\to\widehat{SH}_{G}^{k}(M)$ is given
by $[T\to M\times EG]\mapsto[(T,\emptyset)\to M\times EG]$ and a
natural transformation $\widehat{SH}_{G}^{k}(M)\to DSH^{k+1}(M)$ is given
by $[(T,S\times EG)\to M\times EG]\to[S\to M]$. It is apparent from the bordism
relation in $\widehat{SH}_{G}^{k}(M)$ that this induces a long exact
sequence:
\[
...\to DSH_{G}^{k}(M)\to SH_{G}^{k}(M)\to\widehat{SH}_{G}^{k}(M)\to DSH_{G}^{k+1}(M)\to... .\]

\noindent This proves Theorem 3.

Now we prove that the Tate cohomology groups are an obstruction to the existence of free actions, if the adjoint representation is orientation preserving:

\begin{proof} (Theorem 4) Assume the action on $M$ is free and orientation preserving then the map $\Theta:DSH^k_G(M)\to SH^k_G(M)$ is an isomorphism, since in this case both
sides agree with the singular cohomology of the quotient space by Proposition  14 and Theorem 1. Moreover, $\Theta$ is the map in the long exact sequence from Theorem 3 so $\widehat{SH}_{G}^{k}(M)$ vanishes for all $k$. 

If the action is not free we can find an orbit $G/K$ for some non trivial subgroup $K$. Since $G/K$ might be non orientable, choose any non trivial finite subgroup $H\leq K$. This allows to factorize the map $G/H\to pt$ ($pt=point$) through $M$ via the map $G/H \to M$. Thus we will be finished if we show that the map $\widehat{SH}_{G}^{k}(pt)\to \widehat{SH}_{G}^{k}(G/H)$ is non trivial for some $k$. By definition, the groups $DSH_{G}^{k}(pt)$ and $DSH_{G}^{k}(G/H)$ are trivial for $k>-dim(H)$. Therefore, by the long exact sequences of Theorem 3 it is enough to show that the map $SH_{G}^{k}(pt)\to SH_{G}^{k}(G/H)$ is non trivial for some large $k$. This follows from a theorem by Swan (\cite{Sw2}, Theorem 2) where he shows that the map $f^{*}:H^{k}(BG)\to H^{k}(BH)$ is non zero for infinitely many values of $k$. Here we use Theorem 1 to translate this information to  equivariant stratifold cohomology. 
\end{proof}

\begin{prop}
Let $M$ be an oriented $G$-manifold, and $\Sigma$ the set of all points with non trivial stabilizer. If $U$ is an equivariant open neighbourhood of $\Sigma$ then the restriction map $\widehat{SH}_{G}^{k}(M) \to \widehat{SH}_{G}^{k}(U)$ is an isomorphism. If $\Sigma$ is an oriented submanifold and the action is orientation preserving then the map $\widehat{SH}_{G}^{k}(M) \to \widehat{SH}_{G}^{k}(\Sigma)$ is an isomorphism.\end{prop}
\begin{proof}
The first statement follows from the Mayer-Vietoris sequence for the open cover $U,M \setminus \Sigma$ and the fact that $\widehat{SH}_{G}^{k}$ vanishes for $M \setminus \Sigma, U \setminus \Sigma$ since the action there is free. The second statement follows by choosing $U$ to be an equivariant tubular neighbourhood, so the inclusion of $\Sigma$ in $U$ is a $G$-homotopy equivalence.
\end{proof}

\section{\label{sec:Examples}Applications}

As mentioned in the introduction, one reason for giving geometric constructions for (co)homology theories is  that they sometimes allow computations; for example by giving explicit representatives for (co)homology classes and based on this determining induced maps. The simplest example of this type is the fundamental class of a closed oriented smooth manifold $M$ which is not so easy to construct explicitly as a singular chain (unless one considers the construction of a triangulation an easy task). In stratifold homology this is just the homology class represented by the identity map on $M$. In this spirit we will study in detail the equivariant cohomology of weighted circle actions on spheres. Another example in the same spirit is a geometric description of a product in the cohomology of the loop space $LBG$ of free continuous loops on the classifying space $BG$ of a $G$ defined by Gruher. We will demonstrate the use of our description by a simple example. 

\subsection*{Weighted circle action on a sphere}

~

Consider $S^{2n-1}\subseteq\mathbb{C}^{n}$, for integers $k_{1},k_{2},...,k_{n}$
we define an action of $S^{1}$ by: \[
\alpha\cdot(x_{1},...,x_{n})=(\alpha^{k_{1}}x_{1},...,\alpha^{k_{n}}x_{n}).\]
Denote $K=\Pi k_{i}$ then:

\begin{thm}
$\widehat{SH}_{S^{1}}^{k}(S^{2n-1})=\begin{cases}
\mathbb{Z}/K & k\, even\\
0 & else.\end{cases}$ 

\noindent A generator of $\widehat{SH}_{S^{1}}^{2k}(S^{2n-1})$ is given as follows: If $k\geq 0$  it is given by $[Id\times i_k:(S^{2n-1}\times ES^1_{(k)},\emptyset) \to S^{2n-1}\times ES^1]$ where $ES^1$ is the unit sphere in the Hilbert space, $ES^1_{(k)}$ is the intersection of the sphere with a subspace of complex codimension k, and $i_k$ is the inclusion. If $k<0$ a generator of $\widehat{SH}_{S^{1}}^{2k}(S^{2n-1})$ is given by $[\pi \times Id,\pi:(M_{\pi} \times ES^1, S_f^{-2k-1}\times S^{2n-1}) \to S^{2n-1}\times ES^1]$ where $M_\pi$ is the mapping cone of the projection $\pi:S_f^{-2k-1} \times S^{2n-1}\to S^{2n-1}$ considered as a stratifold with boundary $S_f^{-2k-1} \times S^{2n-1}$, and $S_f^{-2k-1}$ is the sphere with the standard action given by complex multiplication.
\end{thm}
\begin{proof}
Let $S_{f}^{2n-1}$ be the sphere with the standard (free) action given by complex multiplication.
Look at the map \[
g:S_{f}^{2n-1}\to S^{2n-1}\]
given by 

\[
g(x_{1},\dots,x_{n})=(x_{1}^{k_{1}},\dots,x_{n}^{k_{n}})/\left\Vert (x_{1}^{k_{1}},\dots,x_{n}^{k_{n}})\right\Vert . \]
By comparing the Serre spectral sequence for equivariant cohomology for both spheres, together with the isomorphism $SH_{G}^{k}\cong H_{G}^{k}$,
we deduce that: \[
SH_{S^{1}}^{k}(S^{2n-1})=\left\{ \begin{array}{c}
\mathbb{Z}\,\,\,\,\,\,\,\,\,0\leq k<2n\,\, and\, even\\
\mathbb{Z}/K\,\,\,\,\,\,\,\,\,\,2n\leq k\,\, and\, even\\
0\,\,\,\,\,\,\,\,\,\,\,\,\,\,\,\,\,\,\,\,\,\,\,\,\,\,\,\,\,\,\,\,\,\,\,\,\,\,\,\,\,\,\,\,\,\,\,\,\,\,\, else.\end{array}\right.\]
A similar computation can be done for equivariant homology. Then,
using equivariant Poincar\'e duality we get: \[
DSH_{S^{1}}^{k}(S^{2n-1})\cong H_{2n-2-k}^{S^{1}}(S^{2n-1})=\left\{ \begin{array}{c}
\mathbb{Z}\,\,\,\,\,\,\,\,\,\,\,\,\,0\leq k<2n\,\, and\, even\\
\mathbb{Z}/K\,\,\,\,\,\:\,\,\,\,\,\,\,\,\,\,\,\,\,\, k<0\,\, and\, odd\\
0\,\,\,\,\,\,\,\,\,\,\,\,\,\,\,\,\,\,\,\,\,\,\,\,\,\,\,\,\,\,\,\,\,\,\,\,\,\,\,\,\,\,\,\,\,\,\,\,\,\,\,\,\,\, else.\end{array}\right.\]
The only non vanishing map $DSH_{S^{1}}^{l}(S^{2n-1})\to SH_{S^{1}}^{l}(S^{2n-1})$ can occur when $l$ is even and $0\leq l<2n$. In these cases both groups are infinite cyclic.
Consider the following commutative square, where all groups are infinite
cyclic: 

\[
\begin{array}{ccc}
DSH_{S^{1}}^{2k}(S^{2n-1}) & \to & SH_{S^{1}}^{2k}(S^{2n-1})\\
\downarrow &  & \downarrow\\
DSH_{S^{1}}^{2k}(S_{f}^{2n-1}) & \to & SH_{S^{1}}^{2k}(S_{f}^{2n-1})\end{array}\]
The right vertical map is an isomorphism by the Serre spectral sequence, since $g^*:H^l(S^{2n-1}; \mathbb Z)\to H^l(S_f^{2n-1}; \mathbb Z)$ is an isomorphism for $l<2n-1$. 
The bottom map is an isomorphism since the action is free. It
is enough to compute the left vertical map. Since the (Umkehr) map
$DSH_{S^{1}}^{2k}(S_{f}^{2n-1})\xrightarrow{g!} DSH_{S^{1}}^{2k}(S^{2n-1})$ is
an isomorphism (same argument as for the right vertical map), it is enough to compute the composition 

\[
DSH_{S^{1}}^{2k}(S_{f}^{2n-1})\xrightarrow{g!} DSH_{S^{1}}^{2k}(S^{2n-1})\to DSH_{S^{1}}^{2k}(S_{f}^{2n-1}).
\]

\noindent Take a generator $\alpha$ for $DSH_{S^{1}}^{2k}(S^{2n-1}_f)$, it is
represented by the map: \[
h:S_{f}^{2n-2k-1}\to S^{2n-1}_f\]
given by \[
h:(x_{1},\dots,x_{n-k})=(x_{1},\dots,x_{n-k},0,\cdots,0)\]

\noindent The image of $\alpha$ under this composition is represented by the pullback of the following diagram, denoted here by $\tilde{S} \to S^{2n-1}_f $:

\[
\begin{array}{ccc}
\tilde{S} & \to & S^{2n-1}_f\\
\downarrow &  & \downarrow\\
S_{f}^{2n-2k-1} & \to & S^{2n-1} \end{array}\]

\noindent If we forget the action, then the left vertical map is of degree $K$ since the right vertical map is. This implies that $[\tilde{S} \to S^{2n-1}_f]=K\cdot \alpha$.

When $k>0$ the map $SH_{S^{1}}^{2k}(pt)\to SH_{S^{1}}^{2k}(S^{2n-1})$, induced by the map $S^{2n-1} \to pt$ is surjective. The map $SH_{S^{1}}^{2k}(S^{2n-1}) \to \widehat{SH}_{S^{1}}^{2k}(S^{2n-1})$ is also surjective, so a generator of $\widehat{SH}_{S^{1}}^{2k}(S^{2n-1})$ can be taken to be the image of a generator of $SH_{S^{1}}^{2k}(pt)$, as appears in the statement of the theorem.

When $k < 0$ the map $DSH_{S^{1}}^{2k+1}(pt) \to DSH_{S^{1}}^{2k+1}(S^{2n-1})$ induced by the map $S^{2n-1}\to pt$ takes a generator to a generator, so $[S_f^{-2k-1}\times S^{2n-1}\to S^{2n-1}]$ generates $DSH_{S^{1}}^{2k+1}(S^{2n-1})$. Since $k < 0$  the map $\widehat{SH}_{S^{1}}^{2k}(S^{2n-1})\to DSH_{S^{1}}^{2k+1}(S^{2n-1})$ is an isomorphism. The element in the statement is mapped to $[S_f^{-2k-1}\times S^{2n-1}\to S^{2n-1}]$, hence it generates $\widehat{SH}_{S^{1}}^{2k}(S^{2n-1})$.

\end{proof}

\subsection*{A Product in the cohomology of LBG}

~

Let $BG$ denote the classifying space of $G$ and $LBG$ the free loop space on $BG$, i.e. the space of continuous maps from $S^1$ to $BG$. Gruher and  Westerland \cite{G-W} have defined a product (induced from the co-ring spectrum structure on $LBG^{-ad}$) in the cohomology of $LBG$. In Proposition 8.1 they give a description of this product in terms of the exterior product and the Umkehr map. Since the exterior product and the Umkehr map have a simple description in stratifold cohomology, we can simplify this construction.

We consider the conjugation action of $G$ on itself and denote this $G$-manifold by $G^{ad}$. The starting point of the construction of the product is the fiberwise homotopy equivalence $G^{ad}\times_{G}EG\to LBG$ over $BG$ (see the appendix in \cite{KSS}). Using this we identify  the cohomology groups $H^{n}(LBG)$ with $H^n_G(G^{ad})$, which we identify with  $SH^n_G(G^{ad})$ and construct Gruher's product 
\[
*: H^{n}(LBG)\otimes H^{m}(LBG)\to H^{n+m-dim(G)}(LBG),\]
as a product:
\[
*:SH^{n}_{G}(G^{ad})\otimes SH^{m}_{G}(G^{ad}) \to SH^{n+m-dim(G)}_{G}(G^{ad}).\]

Recall our assumption that the conjugation action is orientation preserving. Then we can apply Poincar\'e duality and construct the product in backwards homology. 

The group multiplication in $G$ is an equivariant map $\mu:G^{ad}\times G^{ad}\to G^{ad}$. Thus it induces a product:
 \[
DSH_{k}^{G}(G^{ad})\otimes DSH_{l}^{G}(G^{ad})\to DSH_{k+l}^G( G^{ad} \times G^{ad})\xrightarrow{\mu_*} DSH_{k+l}^{G}(G^{ad})\]
Applying Poincar\'e duality we get the product:
\[
*:SH^{n}_{G}(G^{ad})\otimes SH^{m}_{G}(G^{ad}) \to SH^{n+m-dim(G)}_{G}(G^{ad})\]

Since in stratifold (co)homology Poincar\'e duality is a tautology we can describe the product explicitly:

\begin{prop} Given
elements 

\[
\alpha=[S\to G^{ad}\times EG]\in DSH_{k}^{G}(G^{ad})\,\,\, and\,\,\,\beta=[S'\to G^{ad}\times EG]\in DSH_{l}^{G}(G^{ad})\]
then $\alpha* \beta$ is represented by the composition:

\[
\alpha* \beta = [S\times S'\to G^{ad}\times EG\times G^{ad}\times EG\xrightarrow{\mu}G^{ad}\times EG\times EG\cong G^{ad}\times EG]\]

\end{prop}

The following relates this product to the cup product in $SH^*_G(pt)\cong H^*(BG)$:

\begin{Lemma}
Let $pt$  be a one point space and $i:pt\to G$ given by $i(pt)=1$. The Umkehr map  $i_!:SH^*_G(pt)\to SH^{*+dim(G)}_G(G^{ad})$ is injective and $i_!(\alpha \cup \beta)=i_!(\alpha) * i_!(\beta).$
\end{Lemma}

\begin{proof}
Injectivity follows from the fact that $G\to pt$ is equivariant. By the naturality of the exterior product and the induced maps in $DSH_*$ the following diagram commutes:

\[
\begin{array}{ccc}
SH^n_G(pt) \otimes SH^m_G(pt) & \to & SH^{n+dim(G)}_G(G^{ad}) \otimes SH^{m+dim(G)}_G(G^{ad}) \\
\downarrow &  & \downarrow\\
SH^{n+m}_G(pt) & \to & SH^{n+m+dim(G)}_G(G^{ad}) \end{array}\]

The left vertical arrow is the cup product in $SH^*_G(pt)$ and the right vertical map is the product $*$. The horizontal maps are $i_!$.

\end{proof}

We demonstrate the use of this description by an example. Take $G=S^3$. The bundle $S^3\xrightarrow{j} (S^3)^{ad}\times_{S^3} ES^3\xrightarrow{\pi} BS^3$ has a section. In this case, it implies that the Serre spectral sequence collapses at the $E_2$ page. We get 
\[
H^k_{S^3}((S^3)^{ad})\cong \left\{ \begin{array}{c}
\mathbb{Z}\,\,\,\,\,\,\,\,\,\,\,\,\,\,\,\,k=0,3 \, mod\, 4\\  
0\,\,\,\,\,\,\,\,\,\,\,\,\,\,\,\,\,\,\,\,\,\,\,\,\,\,\,\,\,\, else.\,\,\,\,\,\,\,\,\,\,\,\,\,\end{array}\right.
\]
 
Choose generators for $H^*_{S^3}((S^3)^{ad})$ in degrees $4k$ and $4k+3$, denoted by $a_{4k}$ and $b_{4k+3}$ respectively.

\begin{prop}  For $G = S^3$ the product $*$ is given (up to sign) by: 
\[
a_{4k} * a_{4l} = 0, a_{4k} * b_{4l+3} = a_{4k+4l}, b_{4k+3} * b_{4l+3} = b_{4k+4l+3}
\]

\end{prop}

\begin{proof}  $H^*(BS^3)$ is a polynomial ring with a generator $x_4$ in degree $4$. Represent the additive generator of $H^{4k}(BS^3)$, $(x_4)^k$, by a singular $S^3$-Hilbert stratifold $T_{4k} \to ES^3$ (we can take $ES^3$ to be the unit sphere in the infinite dimensional separable quaternionic Hilbert space with the action of $S^3$ given by quaternionic multiplication, then $T_{4k}$ can be taken to be its intersection with a subspace of (quaternionic) codimension $k$).

The class $a_{4k}=\pi^*((x_4)^k)$, hence is represented by the map $(S^3)^{ad}\times T_{4k}  \to  (S^3)^{ad} \times ES^3$.
$b_{4k+3}=j_!((x_4)^k)$, hence is represented by the map $pt \times T_{4k} \to  (S^3)^{ad} \times ES^3$.
For degree reasons $a_{4k} * a_{4l} = 0$. By  Lemma 20 we get $b_{4k+3} * b_{4l+3} = b_{4k+4l+3}$. 
Since the composition $(S^3)^{ad}\times pt \to (S^3)^{ad}\times (S^3)^{ad} \xrightarrow{\mu} (S^3)^{ad} $ is the identity, we get that $a_{4k} * b_{4l+3}=a_{4k+4l}$ (we ignore signs, since we could choose the generators such that there would not be signs).

\end{proof}

\end{document}